\documentclass[12pt]{amsart}
\usepackage[margin=1in]{geometry}
\usepackage[parfill]{parskip}
\usepackage{amsmath,amsfonts,amscd,amssymb,bm,amsthm}
\usepackage{tikz-cd}
\usetikzlibrary{positioning,arrows,calc}
\usepackage{graphicx}
\usepackage{comment}
\usepackage[pdfencoding=auto,colorlinks,
            linkcolor=red,
            anchorcolor=blue,
            citecolor=green]{hyperref}
\usepackage{bookmark}

\usepackage[backend=bibtex, style=alphabetic]{biblatex}\bibliography{Pi1.bib}
\usepackage[all]{xy}
\usepackage{tikz}
\definecolor{myred}{RGB}{183,18,52}
\definecolor{myyellow}{RGB}{254,213,1}
\definecolor{myblue}{RGB}{0,80,198}
\definecolor{mygreen}{RGB}{0,155,72}
\newcommand{\mS}{\mathcal{S}}

\newcommand{\cod}{\mbox{cod}}

\newcommand{\Ff}{{\mathcal F}}

\newcommand{\eps}{{\epsilon}}
\newcommand{\mT}{\mathcal{T}}
\newcommand{\mU}{\mathcal{U}}
\newcommand{\mC}{\mathcal{C}}
\newcommand{\Diff}{{\rm Diff}}

\newcommand{\tr}{^\top}

\newcommand{\mR}{\mathcal{R}}
\newcommand{\mA}{\mathcal{A}}

\newcommand{\mF}{\mathcal{F}}
\newcommand{\mE}{\mathcal{E}}

\newcommand{\mD}{\mathcal{D}}

\newcommand{\ZZ}{\mathbb{Z}}

\newcommand{\CC}{\mathbb{C}}

\newcommand{\RR}{\mathbb{R}}

\newcommand{\Dd}{{\mathcal D}}

\newcommand{\w}{\omega}
\newcommand{\om}{\omega}

\newcommand{\red}{\textcolor{red}}

\newcommand{\Fol}{{\rm Fol}}

\newtheorem{thm}{Theorem}[section]
\newtheorem{dfn}[thm]{Definition}
\newtheorem{cor}[thm]{Corollary}
\newtheorem{lma}[thm]{Lemma}

\newtheorem{prp}[thm]{Proposition}

\newtheorem{rmk}[thm]{Remark}

\begin{document}
\title{Chambers in the symplectic cone and stability of symplectomorphism group for ruled surface  }
\author{ Olguta Buse and Jun Li }
\address{Department  of Mathematics\\  University of Michigan\\ Ann Arbor, MI 48109}
\email{lijungeo@umich.edu}

\address{Department  of Mathematics\\  IUPUI\\ Indianapolis, IN 46032}
\email{obuse@iupui.edu}

\date{ \today}
\begin{abstract}
We continue the work of \cite{BL1} to prove that for any non-minimal ruled surface $(M,\w)$, the stability under symplectic deformations of $\pi_0, \pi_1$ of $Symp(M,\w)$ is guided by embedded $J$-holomorphic curves.   Further, we prove that for any fixed sizes blowups, when the area ratio $\mu$ between the section and fiber goes to infinity, there is a topological colimit of $Symp(M,\w_{\mu}).$     Moreover, when the blowup sizes are all equal to half the area of the fiber class,  similar to \cite{BL1}, we give a topological model of the colimit which induces non-trivial symplectic mapping classes in $Symp(M,\w) \cap \Diff_0(M).$ These mapping classes are not Dehn twists along Lagrangian spheres.
 \end{abstract}   

\maketitle
\setcounter{tocdepth}{2}
\tableofcontents
\section{Introduction}\label{Intro}

We study some topological aspects of symplectomorphism groups along the line of \cite{Abr98, AM99, McDacs, LP04, ALP, AGK09, Buse11, LL16, LLW16, ALLP}, etc. We will address the topological behavior of the symplectomorphism groups as the form $\omega$  varies within the symplectic cone.  This is a follow-up note of \cite{BL1}, which formulates the symplectic stability and symplectic isotopy conjectures for arbitrary non-minimal ruled surfaces, and establishes them for the one-point blowups of minimal irrational ruled surfaces.   Recall that the stability conjecture informally states that the reduced symplectic cone has a chamber partition such that symplectomorphism groups have homotopy type invariant within the chambers.

As it will be explained in Section \ref{stabproof}, it is not straightforward to extend the result of \cite{BL1} to the multi-point setting.  The key is to control the degeneration of embedded J-holomorphic curves in certain homology classes, so that the classes of their simple representatives form a basis for the symplectic cone.  This is the reason that we only partially establish this conjecture for multi-point blowups.

Let $M_g=\Sigma_g \times S^2$. We will focus on the stability of $\pi_0, \pi_1$ of $Symp(M_g\# \overline{n \CC P^2}, \omega_t)$ while the symplectic form varies in a family such that $\w_t(\Sigma_g) \to \infty$ and  symplectic areas remain the same.

 By McDuff's classification results \cite{McD94}, any symplectic form on $M_g$ is diffeomorphic to $\mu \sigma_{\Sigma_g} \oplus \sigma_{S^2}$ for some $ \mu >0$, up to diffeomorphism and normalization.  This classification result  also holds in the  blowups $M_g\# \overline{n \CC P^2}$ \cite{LLiu2}: if one picks an $\w$ on $M_g\# \overline{n \CC P^2}$, then we normalize $\w$ to have  areas $(\mu, 1, c_1, \cdots, c_n)$ on the homology classes $B, F, E_1, \cdots, E_n$. The vectors  $u=(\mu, 1, c_1, \cdots, c_n)$ determine all possible symplectic  cohomology classes and belong to a convex region $\Delta^{n+1}$ in $\RR^{n+1}$, whose boundary walls are $n$-dimensional convex regions given by linear equations. We will be concerned with symplectic deformations inside such $\Delta^{n+1}$ for the $n$-point blowups.




The topology of symplectomorphism groups is tied with the space of almost complex structures. The core of our paper is to understand the space of almost complex structures. 

Throughout  the paper, we use the notation $G^g_{u,n}$ for  $Symp (M_g\# \overline{n \CC P^2} ,\omega)\cap  
\Diff_0( M_g\# \overline{n \CC P^2})$, where $[\omega]=u$ is a reduced cohomology class (see Section \ref{cone} for the definition)\footnote{Even though $u$ is a just reduced cohomology class rather than an isotopy class, there is no ambiguity as explained in Section \ref{tcinf}.} and $\Diff_0(  M_g\# \overline{n \CC P^2})$ is the identity component of the group of diffeomorphisms.

In particular,  Lemma \ref{incup} and Proposition \ref{inflation} allow us to prove the following:

 \begin{thm} \label{stab01intro}
  Suppose $\mu, \mu'>max\{g, n\}$ for $u= (\mu, 1, c_1,\cdots, c_n)$, $u'= (\mu', 1, c_1,\cdots, c_n)$.  Then the groups  $\pi_0$ and $\pi_1$ of $G^g_{u,n}   $ and $   G^g_{u',n}  $ are  the same.
\end{thm}

\begin{rmk}
The main Theorem can be stated in a larger generality if we fully introduce the simplicial structure of the reduced symplectic cone. Basically, the stated stability holds true for any family of $u= (\mu, 1, c_1,\cdots, c_n)$ as long as these vectors do not escape to the simplicial walls.

This is because on the boundaries of the symplectic cone chambers (see Figure \ref{rule2} for an illustration), some symplectic curves will disappear as their area becomes trivial, and they may change the connectedness and codimension 1 cycles in  $G^g_{u,n}.$

Nevertheless, the way we state Theorem \ref{stab01intro} is a weaker but convenient version,  and it is sufficient for Proposition \ref{tlimitintro}.
\end{rmk}

 We remark that Lemma \ref{inflation} grants us that there is a topological colimit on horizontal lines inside the reduced symplectic cone as  $\mu \to \infty.$

 The absence of certain badly-behaved almost complex structures strata allow us to establish the full stability conjecture in the particular case when the blow-ups are all equal to half of the area of the fiber $[S^2]$:

 \begin{thm} \label{stabintro}
The homotopy type of $G^g_{\mu,n}$ is constant for $\frac{k}{2} < \mu \le \frac{k+1}{2},$ for
any integer $k \ge 2g$. Moreover as $\mu$ passes the half integer $\frac{k+1}{2}$, all the groups
$\pi_i, i = 0,\cdots, 2k + 2g - 1$ do not change.
\end{thm}

In Chapter \ref{s:out},  we proceed to describe topologically the colimit $G^g_{\infty,n}$ of the symplectomorphisms groups in the equal size $\frac12$ blow-ups.
We establish the following theorem on the disconnectedness of this topological model denoted by $\mD^g_n$.  For a more detailed description of $\mD^g_n$ and $G^g_{\infty,n}$, see Definition \ref{fibergp}.

\begin{prp}\label{tlimitintro}
  Take $M_g\# \overline{n \CC P^2}$ with forms in classes $(\mu,1, \frac12,\cdots,\frac12)$. Then there is a smooth topological model group $\Dd^g_n$ so that:
\begin{enumerate}
\item $\Dd^g_n$ is weakly homotopic to the topological colimit  $G^g_{\infty,n}$ obtained as $\mu$ goes to $\infty$.

\item The group $\Dd^g_n$ is disconnected when $g\ge 2$.

\item When $\mu\to \infty$, for $i=0,1$, $\pi_i(G^g_{u,n})=   \pi_i(G^g_{\infty,n})$  and hence the groups $G^g_{u,n}$ are disconnected for $g \geq 2$.
\end{enumerate}

\end{prp}

Notice that the above theorem extends the results of \cite{BL1} from the one-point blowup of a minimal ruled surface to arbitrary points blowups.  It remains unknown what the full group $\pi_0 \mD^g_n$ is, and in particular whether the disconnectivity in Proposition \ref{tlimitintro} holds for other symplectic forms in horizontal lines in the cone.  We hope to explore these questions in future works.




On the line of equal size $1/2$, the rank of $\pi_1(G^g_{\mu,n})$ is positive (indeed at least $n$)  for an $n$-fold blowup of $M_g$.   On the other hand, by \cite[Corollary 3.6]{HK15}, if the number of blowups is large ($n \ge \mu/2$), there are points on this line that do admit Hamiltonian circle actions.  Note that this never appears in the minimal cases or their one-point blowups since these cases all admit Hamiltonian circle actions.

As a corollary, we find that 

\begin{cor}
 There are new elements in $\pi_1(G^g_{u,n})$ for a symplectic form without any Hamiltonian circle actions.
\end{cor}

Notice that previously those kind of elements were discussed in \cite{Kedra} for general type surfaces and \cite{ABP19} for certain forms on rational 4-manifolds.

\noindent\textbf{Acknowledgements.} We are grateful to  Richard Hind, T-J Li, Dusa McDuff, Weiwei Wu, Michael Usher and  Weiyi Zhang for helpful conversations.

\section{Preliminaries}\label{cone}
This section covers the preliminary results needed to develop the proof strategies. These results have been largely described in \cite{BL1}. We include them here to make the current manuscript self-contained.

\subsection{The symplectic cone}

By \cite{LL01}, if $(M,\w)$ is a closed  symplectic 4-manifold with $b^+= 1$, the symplectic canonical class is unique once we fix $u=[\w]$, and we denote it by $K_u$.

Let $\mE$ be the set of exceptional sphere classes, and  $\mC_M \subset H^2(M,\RR)$ denotes the symplectic cone, the collection of symplectic form classes.

In Theorem 4 of \cite{LL01}, Li-Liu showed that if $M$ is a closed, oriented 4-manifold with $b^+= 1$
and if the symplectic cone  $\mC_M$ is nonempty, then
$$\mC_M= \{u\in P |0 < u \cdot E \quad \text{for all} \quad E \in \mE \},$$
where $P\subset H^2(M,\RR) $ is the cone of positive form classes.

\begin{dfn}
 For a cohomology class $\w$ on $M_g\# \overline{n \CC P^2}$,  we normalize $\w$ to have  areas $(\mu, 1, c_1, \cdots, c_n)$ on the homology classes of the standard basis $B, F, E_1, \cdots, E_n$.
  Then the symplectic forms in this cohomology class are isotopic.  cf.  \cite{McD96, LL01}
 Furthermore we define {\it reduced}  vectors $u=(\mu, 1, c_1, \cdots, c_n)$  to be those  with  $\mu>0, c_1+c_2\ne 1, 0<c_i<1,$ $ c_1\geq c_2\geq \cdots \geq c_n.$

 \end{dfn}

We remark that the reduced symplectic classes give a set of orbits under the action of the diffeomorphism group. Since diffeomorphic forms have homeomorphic symplectomorphism groups, it is sufficient to study the stability conjecture for reduced classes. 
 We will be concerned with symplectic deformations inside a  convex region $\Delta^{n+1}$ in $\RR^{n+1}$,  the right side of the reduced symplectic cone, given as a linear truncation with $\mu \geq 1$. We refer this region as the symplectic cone for the brevity of writing.  Strictly speaking, $\Delta^{n+1}$ is a slice of $\mC_M$ where the fiber class has area 1.

Note that the way we partition this symplectic cone is by looking at the homology classes of potential symplectic curves. To that end, we will now fix some notation:

\begin{dfn}\label{sw}
Let $\mathcal S_{\omega}$ denote the set of homology classes of embedded $\omega$-symplectic curves and $K_{\w}$ the symplectic canonical class. For any $A\in \mathcal S_{\omega}$, by the adjunction formula,
  \begin{equation}\label{AF}
    K_{\omega}\cdot A=-A\cdot A -2+2g(A).
\end{equation}

For each $A\in  \mathcal S_{\omega} $ we associate the integer
$$ {\cod_A}=2(-A\cdot A-1+g).$$

\end{dfn}

 We can define $\mS_u$, where $u=[\w]$ accordingly, only using the cohomology data of $\w$.   We are going to denote $\mS^{<0}_u$ by the subset of $\mS_u$, which is the classes having positive codimension, which, if represented by pseudo-holomorphic curves, correspond to curves of negative indexes.

\begin{rmk}
\begin{itemize}

\item By Lemma 2.4 in \cite{ALLP}, $\mS_u=\mS_w$.

\item By the main theorem in \cite{LU06}, the negative self-intersection classes that admit embedded representatives are exactly those that have positive pairings with the class $\w.$ Namely, we can find some integral class $u'$ that admits some embedded curve in those classes of  $\mS^{<0}_u$, and the symplectic inflation of \cite{LU06} allows us to change the class $u'$ into $u.$
 \end{itemize}
\end{rmk}

Let us explain the simplicial structure of $\Delta^{n+1}$ that guides the stability conjecture of the symplectomorphism group.

\begin{dfn}

An {\bf extremal exterior wall} (not contained in the cone) of $\Delta^{n+1}$, is given by $\{u| u\cdot E=0\}$ where $E\in \mE$ is any exceptional class. 

Lastly, a {\bf reduction exterior wall} (contained in the cone) of $\Delta^{n+1}$, is given by
$\{u| u\cdot D=0$ where $D$ is one of the square $-2$ homology classes given by the reduced condition in $\mR=\{ F-E_1-E_2, E_i-E_j, i>j>0\}$.

A  partition of the symplectic cone $\Delta^{n+1}$ into {\bf chambers}  is given by  separations with {\bf interior walls} given by  $u\cdot A = 0,$ where $A\in \mS^{< 0}_u \setminus (\mE\cup \mR).$

\end{dfn}
It is important to point out that for the purpose of proving the current results of stability in the first two homotopy groups, we will only need to partition with interior walls where the classes $A$ are of {\it section type }   ${B+kF-\sum_i c_iE_i}$.

Here we draw the picture for  $\Delta^3$ (with chambers when $\mu \ge 1$) of a two-point blowup of $\Sigma_g\times S^2:$

 \begin{minipage}{.5\textwidth}
  \centering
\includegraphics[height=2.5in]{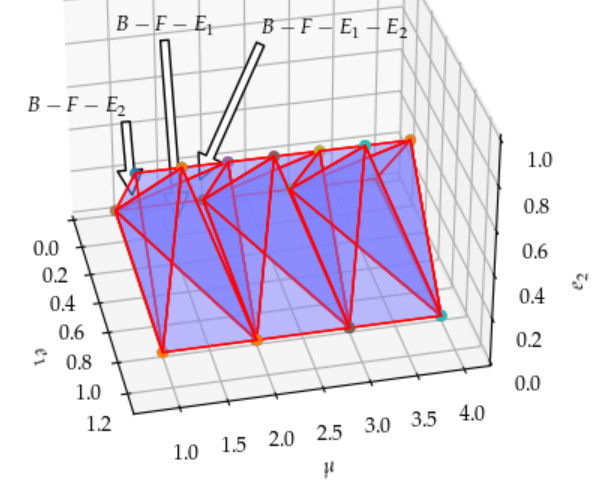}
  \label{rule2}
    \end{minipage}%
    \begin{minipage}{.5\textwidth}
Notice that in this figure, the $Y-$axis means the area ratio $\mu,$  $X$ and $Z$ are blowup sizes of $E_1, E_2$ respectively.  Clearly, $Y$ goes to $\infty.$     The red triangles are the walls defined by the curve classes, and the tetrahedron is stable chambers in Theorem \ref{stab01intro}.
    \end{minipage}%

\subsection{The homotopy fibration and the inflation strategy}
Overwhelmingly papers discussing the topology of spaces of symplectomorphism groups in the ruled surfaces settings base the results on the following fibration first initiated by Kronheimer\cite{Kro99}  and then adapted by McDuff\cite{McDacs}.

McDuff's approach in  \cite{McDacs} was to consider  Kronheimer's fibration in  \cite{Kro99}:

\begin{equation}\label{fib}
  Symp(M, \w)\cap \Diff_0(M) \to \Diff_0(M) \to \mT_{\w},
\end{equation}

where $ \mT_{\w}$ represents the space of symplectic forms in the class $[\w]$ and isotopic to a given form, and $\Diff_0(M)$ is the identity component of the diffeomorphism group. Moser type results provide a transitive action of $\Diff_0(M)$ on  $ \mT_{\w}$ and gives  the fibration \ref{fib}.

To facilitate maps between such fibrations as $\omega$ is deformed, McDuff'  \cite{McDacs}\footnote{ McDuff's original results were written in terms of homotopy fibration where the larger space of taming almost complex structures was used. Using the fact that the space of taming almost complex structures is homotopy equivalent to the one we use of compatible structures, for reasons explained in Section \ref{tcinf} we will use the compatible structure spaces.}  uses the (weak) homotopy equivalence between $ \mT_{\w}
$ and the space $\mA_{\w}$ (which is the space of $\w'$-compatible almost complex structures, where $\w'$ is any symplectic form isotopic to $\w$) to construct a homotopy fibration 
\begin{equation}\label{fibacs}
    Symp(M, \w)\cap \Diff_0(M) \to\Diff_0(M) \to \mA_{\w}.
\end{equation}

This allows one to use inflation techniques described in Section \ref{tcinf}  to move by inclusions the (stratified) spaces $\mA_{\w}$  as the symplectic form varies.

Thus, to succeed in stability results one needs  to establish such (a) existence of strata (b) sufficient embedded or nodal $J$ holomorphic curves for {\it every $J$} in order to move these strata using inflation.

Inflation of symplectic forms that tame a given almost complex structure we first introduce by McDuff \cite{McDacs} and later by Buse \cite{Buse11}.

As discussed in \cite{ALLP} Section 4.4,  their proofs made the unwarranted assumption that for
every $\w$-tame $J$ and every $J$-holomorphic curve $C$ one can find a family of normal planes
that is both $J$ invariant and $\w$-orthogonal to $TC$. This is true only if $\w$ is compatible with
$J$ at every point of $C.$
A correct version of their statements, which works with the proof provided  in \cite{McDacs} for positive self-intersection curves and in Theorem 1.1 in \cite{Buse11} for negative self-intersection curves, combines as:

  \begin{thm}  For a 4-manifold $M$, given a compatible pair  $(J,\w),$ one can inflate along a $J$-holomorphic curve $Z$, so that there exist a symplectic form $\w'$ taming  $J$ such that $[\w']= [\w]+ t PD(Z), t\in [0,\mu)$ where $\mu= \infty$ if $Z\cdot Z\ge0$ and $\mu= \frac{\w(Z)}{(-Z\cdot Z)}$ if $Z\cdot Z<0$.
\end{thm}

 As further explained in \cite{ALLP}, for four-dimensional manifolds $M$ with $b^+=1,$ (such as the blow-up ruled surfaces) the strategy is to perform the above inflation to obtain a $J$-tame $\w'$ in the correct cohomology class. Then using the result of  Li and Zhang \cite{LZ11} giving the comparison of tame/compatible cones, one has an  $\w''$ compatible with the given $J$. Anjos et all call this {\bf $b^+=1$ $J-$compatible inflation}.
 
Note that in the non-minimal case, cohomologous forms are not known to be isotopic, in contrast to the minimal case.
We'll consider a larger space ${\mA_{[\omega]}}$, which is the space of almost complex structures taming cohomologous forms. The space $\mA_{[\omega]}$ is the union of its mutually homeomorphic components given by the isotopy classes of symplectic forms in the cohomology class $[\omega]$. We showed in \cite{BL1}  that one can  keep track of isotopy classes (hence the components) during the $b^+=1$ $J-$compatible inflation with $J$ in the larger space $\mA_{[\omega]}$.

Namely, the inflation process does not change connected components.

Moreover, as explained in \cite{BL1} there is a natural correspondence between the connected components of the spaces $\mA_{[\omega]}$ and the spaces of all cohomologous symplectic forms $\mT_{[\omega]}$ and they are mutually homotopy equivalent.

If we pick one such pair of components $(\mT_{\omega}, \mA_{\omega})$ we have a  the homotopy fibration $G_{\w}\to \Diff_0(M) \to \mA_{ \w}$.

Additionally, as soon as one successfully inflates to establish an inclusion of components of almost complex structure spaces for different cohomology classes $[\omega]$ and $[\omega']$ the following diagram can be used to establish homotopy maps between the spaces of symplectomorphisms groups.

  \begin{equation}
\begin{CD}
G_{\w} @>>> \Diff_0(M)@>>> \mA_{ \w}\\
@VVV @VVV @VVV \\
G_{\w'} @>>>  \Diff_{0}(M)   @>>>  \mA_{ \w'}.\\
\end{CD}
\end{equation}

\subsection{Picking a canonical component of the spaces of almost complex structures in the case $\mu> max\{g,n\} $}

  Denote by $J_{std}$ the integrable structure obtained by consecutive complex blow-ups of the split complex structure on $\Sigma_g \times S^2$.
  First note that $J_{std}$ tames (is compatible with) some symplectic form,  when $\mu$ is very large. This is because for $J_{std}$, this is because we are able to choose the blowup points on a section in class $B$.  The divisors in the blowup are all strict transforms from the minimal model, and by Nakai-Moishezon we are able to find an ample class in $u=(\mu, 1, c_1,\cdots, c_n)$.
  
  We will show that  by performing J-tame (or compatible) inflation, any $J$ in the open stratum $\mA^{top}_u$ as in definition  \ref{subsetA} is compatible with some  symplectic form in $u'=(\mu', 1, c_1,\cdots, c_n)$, where $\mu'>max \{g,n\}$.  This means that $J_{std}$ tames (is compatible with) some symplectic form in any cohomology class when $\mu>max \{g,n\}$.  Hence  using the connected component that contains $J_{std},$ we have a canonical choice of connected components in each $\mA_u$, and by abuse of notation we will from now on call $\mA_u$ those canonical component.

\subsection{Decomposition of $\mA_u$ via pseudo-holomorophic curves}
In what follows we recall the general strategy to define  subsets of the space of almost complex structures $\mA_{u}$. We will adapt this general stratification to our needs for the current results in Section \ref{strA}.

We first define a driven subset of $\mA_{u}$, labeled by set $\mC\subset \mathcal S^{< 0}_u$ for a given isotopy class of $\w$ as following:

 \begin{dfn} \label{fine decomposition}
  A  collection  $\mC\subset \mathcal S^{< 0}_u$ is called admissible if
  $$\mC=\{A_1, \cdots, A_i,\cdots ,A_q |\quad  A_i\cdot A_j \geq 0, \quad \forall  i\neq j\}.$$ Given an admissible collection $\mC$, we  define the real codimension
  of the labelled collection ${\mC}$ as  $$\cod({\mC})= \sum_{A_i\in \mC} \cod_{A_i}=\sum_{A_i\in \mC}  2(-A_i\cdot A_i-1+g_i).$$

  Define the   $\mC$-driven {\bf  subset}
  $$\mA_{u, \mC}:=\{ J\in \mA_{u}|  A\in \mathcal S_{u}^{<0} \hbox{  has an embedded $J$-hol representative if }   A\in \mC\}.$$
 
  And if $\mC=\{A\}$ contains only one class $A$, we will use $\mA_A$ for $\mA_{\{A\}}$

 \end{dfn}

 The following Proposition  \ref{stratum} and Corollary are Proved in \cite{BL1}.

\begin{prp} \label{stratum} 

Let $(X,\omega)$ be a symplectic 4-manifold. 
 Suppose $U_{\mC}\subset\mA_{u}$ is a subset
characterized by the existence of a configuration of embedded
$J$-holomorphic  curves $C_{1}\cup C_{2}\cup\cdots\cup C_{N}$ of  positive codimension as in Definition \ref{sw} with $\{ [C_{1}], [C_{2}],\cdots , [C_{N}]\}=\mC$.
Then $U_{\mC}$ is a co-oriented Fr\'echet suborbifold
   of $\mA_{u}$ of (real) codimension $2N-\sum g_i-2c_{1}([C_{1}]+\cdots+ [C_{N}]).$   
   
   Note here $c_1(\w)=-K_\w$ and by adjunction $2g-2=KA+AA$, we have the following two equivalent way to write the codimension $Cod=\sum_i K\cdot [C_i]-[C_i]^2,$ or $Cod= 2(-C_i\cdot C_i-1+g_i),$ as in Definition \ref{fine decomposition}.

   In particular, this statement covers the case where $(X,\omega)$ is a ruled surface and $C_i's$ all have negative squares.
\end{prp}

\begin{cor}
  The  subsets $\mA_{u, \mC}$ are suborbifolds of $\mA_{u}$ codimension of $Cod(\mA_{u,\mC})$.
\end{cor}

\subsection{Identifying strata  of various  spaces of almost complex structures using inflation}\label{tcinf}

\section{Decompositions of $\mA_u$ and stability under inflation}\label{stabproof}

In this section, we prove the Theorems \ref{stab01intro} and \ref{stabintro} by inflating along embedded or nodal J-holomorphic curves.

\subsection{Curves}

  Let us first recall several results from \cite{Zhang16}. Using our basis notation $B, F, E_1, \cdots, E_n$ of $H_2(X,\ZZ)$,  here are some of Zhang's results that we will use.

For a fixed $J$, let  $\mathcal M_D$ denote the moduli space of subvarieties ($J$ holomorphic images) in class $D$. 

\begin{thm}[Theorem 1.1 or 3.4 in \cite{Zhang16}]\label{intro1}
Let $(M, \omega)$ be an irrational symplectic ruled surface with a base of genus $g$ and $E$ an exceptional class. Then for any  $J$ tamed by $\omega$ and any subvariety in class $E$, each one of its irreducible components is a rational curve of negative self-intersection. Moreover, the moduli space $\mathcal M_E$ is a single point.
\end{thm}

\begin{thm}[Theorem 1.2 in \cite{Zhang16}]\label{curvesexist}
Let $(M, \omega)$ be an irrational ruled surface of base genus $g\ge 1$. Then for any tamed $J$ on $M$,
\begin{enumerate}
\item Through  any given point there is a unique subvariety in the positive fiber class $F$.
\item The moduli space $\mathcal M_F$ is homeomorphic to $\Sigma_g$, and there are only finitely many reducible subvarieties appearing as images of singular J-holomorphic curves in class $F$.
\end{enumerate}
\end{thm}

Here a rational curve means  that  the domain genus is 0, and by the adjunction formula each irreducible image  component has to be smooth, hence one can pick a simple $J$ holomorphic map into it that is embedded:

\begin{lma}\label{gj=0}
The irreducible components of the stable curves of a Gromov limit for an exceptional curve all have $g_J=0$.

\end{lma}

\begin{proof}
 First, by Corollary 3.8 of \cite{Zhang16}, every irreducible rational curve $C_i$ belongs to a fiber in homology class $F$.  This means an irreducible component of the exceptional stable curve of class $E$ is also a component of some curve in class $F$.
 
 Note the the fiber class $F$  pairs non-negatively with any J-holomorphic
subvariety.   Theorem 1.4 of \cite{LZ15}  says that
 $0= g_J(F) \ge \sum_i g_J(C_i) $, and each $g_J(C_i)\ge 0.$  Hence all $g_J(C_i)=0.$
\end{proof}

On the other hand, for the high genus section type curves, Zhang provides the following existing result:

\begin{prp}[Proposition 3.6 in \cite{Zhang16}]\label{smoothsection}For any $(M,\omega)$ a symplectic irrational ruled surface and any  $J$ on it tamed by the form $\omega$ there is an embedded $J$-holomorphic curve $C$ of genus $g$ such that $[C]\cdot F=1$.
\end{prp}

This grants that if $J$ tames a symplectic form in class $u= (\mu,1,c_1, \ldots c_n) $  then there is an embedded curve in the class  ${B+kF-\sum_i c_iE_i}$ for some $k \leq g,$ as long as the curve has positive area.

\subsection{Stratification of $\mA_{u}$}\label{strA}

In the current paper, we are not able to guarantee the existence of enough almost holomorphic curves for every $J$ to establish full stability. As we will see below, the deficiency stems from the lack of embedded exceptional curves for {\it every} $J$ to guarantee inflation within the chambers of the symplectic cone.  However,  embedded or mildly degenerated exceptional $J$- curves are provided for each $J$ in the codimension 2 strata and the open stratum.

Nevertheless, for equal size $\frac12$ blowups (see section \ref{1/2}), we are able to show that the above-mentioned types of strata along with codimension higher than $4$ strata corresponding to negative index embedded curves in classes $B+kF, k<g$, completely cover $\mA_u$. This allows us to conclude that our current inflation results are sufficient to prove full stability in such deformations.

For any almost complex structure $J$ and any exceptional homology class $E$, we know there exists a  $J$ holomorphic curve in the class $E$. Based on this we make the following definition:

\begin{dfn}
For any $J$ and an exceptional class $E$ in a ruled surface, we call the J-holomorphic exceptional curve in class $E$

\begin{itemize}
    \item {\bf mildly degenerated} if its J-representative has exactly two rational embedded components in classes $\{C, D\}$, where $D$ has square $-2$ and $C$ is an exceptional class with $C^2=-1$, and $C \cdot E=0.$

    \item {\bf badly degenerated} if its J-representative either has at least two components with a simple rational representative of square at most $-2$, or has at least one with simple rational representative of square less than $-2$.
\end{itemize}

\end{dfn}
For the mildly degenerated curves, we extend the definition of driven sets   
  $$\mU_{u, C,D}:=\{ J\in \mA_{u}|   \hbox{ J admits a mildly degenerated exceptional curve  of type \{C, D\}} \}$$
  
  This is a codimension 2 suborbifold by  Proposition \ref{stratum}.
  
The following arguments imply that exceptional class curves of type $E_i$ or $F-E_j$, as well as their specific degenerations $\{C, D \}$, are determined by homological data only.   For the rest of the almost complex structures, we will show that they are included in  Fr\'echet orbifolds with codimension 4 or higher.

\begin{lma}\label{estable2}
Consider any exceptional homology class $E$ and any $J \in \mA_u$  on an irrational ruled surface.
One of the following three situations take place:

\begin{itemize}
    \item The class $E$ has an embedded rational $J$ -curve representative.
    \item The stable $J$ holomorphic curve in class $E$ is mildly degenerated.

    \item The stable J- holomorphic curve is badly degenerated.
\end{itemize}

\end{lma}

\begin{proof}
  Consider the stable curve in an exceptional class, $E=\sum_i r_iC_i$ where each $C_i$ is possibly multiple covered.  Let $g_J(A)$ be the virtual genus of class $A$, given by $\frac{A\cdot A +K\cdot A}{2}+1$.   By Lemma \ref{gj=0}, $0=g_J(E) \ge \sum_i r_i g_J(C_i) \ge 0$.  Hence $ g_J(C_i)=0$ for each $C_i$.   Then let $g(\Sigma_i)$ denote the domain genus of the corresponding map into $C_i$, we always have $ g_J(C_i)\ge g(\Sigma_i) $ where $``="$ holds iff the image almost holomorhic curve $C_i$ is smooth.
  That implies that we can choose a simple embedded  $J$ holomorphic map with the image $C_i$ for each $i$  since indeed $g_J(C_i)=0=g(\Sigma_i)\ge 0$.
  
  First, let us consider the set  $\mA_E$ consisting of all $J_0$ admitting embedded exceptional curves. This is an open subset since the curve is stable and for $J$ sufficiently close to $J_0$ we must also have embeddedness of the exceptional curve.

  Let us now look at the rest of $A_u$, consisting of $J$ with degenerated exceptional curves.
  
  Then by the connectedness of the stable image curve, there must be at least one embedded rational component with sufficiently negative self-intersection $-k$, $k\geq 2$. 
  We will be using the virtual dimension computation   (see Theorem 1.6.2 of \cite{IS99} by Ivashkovich-Shevchishin) for underlying simple representatives.
  Namely, for a simple class $A$ admitting an embedded $J$ holomorphic representative and for $K$ the canonical class, index of the curve representing $A$ is  given by 
  
  \begin{equation}\label{indA}
     index(A)= 2g-2-2K\cdot A.  
  \end{equation}

 If there is only one such negative component with square strictly less than $-1$ which is a simple class with self-intersection exactly -2, then the degeneration has to be of $\{C, D\}$ type by computing the square and pairing with the canonical class $K$.
      More precisely, assume $E=C+\sum D_i+\sum P_j,$ so that $C^2=-2, D^2_i\ge -1, P^2_j\ge 0.$   By $g_J(C)=g_J(D_i)=0,$ we have $K\cdot  C=0, K\cdot D_i=-1$ and $K\cdot P_J>-1$.   Also, we have $ K \cdot (C+\sum D_i+\sum P_j) =K\cdot E=-1.$  Hence there can only be precisely one $D_i$.  Furthermore  since the  component  $C$ is embedded then all $J$ admitting  such curve with such particular component   belong to the Cod=2 stratum (see the results in Section 2) and \eqref{indA}.
  
 To see $C\cdot E=0$   when $E=C+D$,  one only need to note the following:   $-1=E^2= (C+D)(C+D)=  C(C+D) +D(C+D) =C\cdot E +  1-2.$   Hence $C\cdot E$ has to be $0.$

 On the other hand, if there is either (a) at least one component of negative self-intersection less or equal to $-3$ or (b) more than one component of self-intersection $-2$ we argue by Ivashkovich-Shevchishin's index formula \eqref{indA} and transversality of strata that the corresponding $J$ belong to higher codimension strata detailed in Sec \ref{cone}.

    In the first case (a), since  there are irreducible components with square less or equalthan $-3$, we write 
 $E=C_1 +\sum_{i>1} C_i,$    such that $C_1^2<-2.$  If $C_1$ has a simple representative, then we are done.  Now let's deal with the case $C_1$ is multiple covered.  Let $C_1=m C'_1, m>1$ such that $C'_1$ has a simple representative.   Notice that we immediately know that  $ (C'_1)^2\le -1.$

      Then we have $0= 2 g_J (C_1) = 2+ m^2 (C'_1)^2 +m K_J \cdot C'_1.$  This  means that  $$K_J \cdot C'_1 =\frac{-2-m^2 (C'_1)^2}{m}>0.$$
Hence $C_1'$ has an index less than -2. This means that $J$ belongs to a stratum of codimension greater than 2 in $\mA_{u}.$

If there are more than one components with square   $-2$ let's write    $E=C_1 +C_2 +\sum_{i>2} C_i,$    such that $C_1^2=C^2_2=-2.$   If both of them have  simple representatives, then we are done.  Now lets assume some of them are multiple covered, i.e.  $C_1=p C'_1,  C^2= q C'_2, p, q \ge 1.$   Now we still have $0= 2 g_J (C_1) = 2+ p^2 (C'_1)^2 +p K_J \cdot C'_1.$  This  means that  $$K_J \cdot C'_1 =\frac{-2-q^2 (C'_1)^2}{q}>0.$$  Similarly, the index of $C'_2$ is also non-positive.  Hence both simple representatives has non positive indices. By the transversality of the simple representatives,  $J$ belongs to a strata of  codimension greater than 2 in  $\mA_{u}.$
 \end{proof}

The previous lemma yields three-partitions of the space $\mA_ u$ that are $E$ dependent. We will use   refinements of  these partitions by intersections with strata given by section type classes as well and as a consequence give a stratification $\mA_{u}$ as follows:
\begin{dfn}\label{subsetA}
The following is a stratification of  $\mA_{u}$ validated by the Lemma \ref{strsep} below.
\begin{itemize}

\item We call  $\mA^{top}_{u}$ to be the collection of $J\in \mA_{u}$ characterized by
 \begin{itemize}
\item [(a)]  there are no   J-holomorphic curves in $B+kF-\sum E_i$, index less or equal -2, and

\item [(b)] an embedded exceptional curve in each class of type $E_i$ or $F-E_i$.

\end{itemize}

$$  \mA^{top}_u = (\underset{E' \in \mE}{\cap} \mA_{E'})  \setminus ( \underset{ind (B+mF-\sum E_i) \le -2}{\cup}\mA_{ B+mF-\sum E_i})  $$ where 
$ B+mF-\sum E_i$ has index no larger than -2.

\item $\mA^{2}_{u}$, is the disjoint union of  codimension two strata that could be of one of the following types:
\begin{itemize}
    \item Degeneration type, forming the collection $\mA_{u}^{mild} $. For 
    a given degeneration  $C,D$  of a given exceptional class classes $E$, we introduce:
    
     $\mA^{2}_{u,C,D}:=(\mU^{2}_{u, C, D} \cap (\underset{E' \ne E}{\cap} \mA_{E'}) )\setminus ( \underset{ind (B+mF-\sum E_i) \le -2}{\cup}\mA_{ B+mF-\sum E_i})  $ where $E'$ is an exceptional class. In words each $J\in\mA^{2}_{u,C,D}$   is characterized by 
     \begin{itemize}
         \item [(a)]there are no   J-holomorphic curves in $B+kF-\sum E_i$, index less or equal -2, and 
     \item [(b)] the existence of exactly one exceptional class having a stable $J$ representative of homology type $\{C, D\}$, s.t. $C^2=-2, D^2=-1$ as in Lemma \ref{estable2} and 
     \item[(c)] all other exceptional classes  of type $E_i$ or $F-E_i$ are represented by embedded J-holomorphic curves.
     \end{itemize}

    \item  The second collection (possible) $\mA^{2}_{u, B}$ is a union of sets obtained as follow.  Pick a class  $B'=B+kF-\sum E_i$  with index $-4k +2g-2 +2l=-2$ where $l$ is the number of $E_i$'s.  The  typical strata given by this $B'$ inside the larger collection is

   $\mA^{2}_{u, B'}:= \mA_{u,  B'}\cap (\underset{E' \in \mE}{\cap} \mA_{E'})  \setminus ( \underset{ B+mF-\sum E_i \ne B'}{\cup} \mA_{ u, B+mF-\sum E_i,  ind \leq-2}).$
     
     In words, this set is characterized by  $J$ for which there exist 
     \begin{itemize}
     \item [(a)] one embedded section in class $B'$ \item[(b)]
     no other embedded representative sections of index less or equal than $-2$  and \item[(c)]
     an embedded exceptional curve in each class of type $E_i$ or $F-E_i$. 
     \end{itemize}

\end{itemize}

\item  Finally, $  \mA^{high}_{u} := \mA_{u}\setminus (\mA^{top}_u \cup \mA^{2}_{u}) $
is the complement of the above two subsets.

\end{itemize}

\end{dfn}

 Notice that Lemma \ref{type2} will guarantee the existence of the strata $\mA^{2}_{u, B}$.

The following table summarizes this definition and is explained in Lemma \ref{strsep}.

\begin{table}[ht]
\begin{center}
\resizebox{\textwidth}{!}{
  \renewcommand{\arraystretch}{2}
 \begin{tabular}{||c c c c ||}
\hline\hline
section class  &   embedded exceptional       &
mildly deg& deg\\[0.5ex]
\hline\hline
No $B+mF-\sum E_i, ind \le -2$  &         $\mA^{top}_{u}$   &
$\mA_{u}^{mild}$             &
$\mA^{high}_{u}$ \\[0.5ex]
\hline
  $B+kF-\sum E_i, ind =-2$   &     $\mA^{2}_{u, B}$ &
$\mA^{high}_{u}$ &   $\mA^{high}_{u}$    \\[0.5ex]
\hline
  $B+kF-\sum E_i, ind <-2$   &     $\mA^{high}_{u}$ &
$\mA^{high}_{u}$ &   $\mA^{high}_{u}$    \\[0.5ex]
\hline\hline
\end{tabular} }
\caption{Partition}\label{partable}
\end{center}
\end{table}

\vspace{10mm}
\smallskip

The following Lemma shows that the  Definition \ref{strsep} is indeed correct and it further clarifies what type of almost complex structures fall in the third set  $\mA^{high}_u$.

\begin{lma}\label{strsep}
$\mA_u$ is the disjoint union of the above 3 parts: 
\begin{enumerate}
    \item $\mA^{top}_{u}$ is an open subset of  Cod=0 in $\mA_u$,
 \item  $\mA^{2}_{u}$  has  $Cod=2$  in $\mA_u$,
 \item $\mA^{high}_u$  is the union (not necessarily disjoint) of (a) strata of $Cod>2$ corresponding to either homology classes of type $B+kF-\sum E_i,$  or to homology classes of square strictly less than -2 of rational curves that appear as a degeneration of the exceptional class or (b) transverse intersections of certain  strata of codimension $2$.
\end{enumerate}

\end{lma}

\begin{proof}

A simple area argument shows that only finitely many open set $\mA_{u,E}$  appear in the Definition \ref{subsetA} and for a given $u$ only finitely many closed sets appear in the unions  $\cup_i\mA_{u, B+mF-\sum E_i,  ind \leq-2}$. This yields part (1) immediately.
By Lemma \ref{smoothsection}, there is always an embedded curve in the section class. 
The proof of part $(3)$ part relies on the way we define the set $\mA^{high}_u$ (see the table \ref{partable}).

Indeed for a $J$ in either one of the sets in the bottom row of the table, there is an embedded $J$-holomorphic curve in the class $B+kF-\sum E_i$, with the index   $-4k +2g-2 +2l$ at most -4. Then $J$ belongs in some  stratum  $\mA_{u,B+kF-\sum E_i}$ of codimension at least $4$  generated by a section class.

Secondly, if a $J$ is part of $\mA^{high}_u$ by being one of the two sets in the second row of the table, then $J$  belongs to a transversal intersection of strata of codimension $2$ given by either two different  different  section type classes with index $-2$  each or one section type class and one mildly degenerated type $C,D$ exceptional curve.

If $J$ admits at least one badly degenerated exceptional curve or at least two mildly degenerated ones, they provide strata (or transverse intersection of such) of codimension at least $4$. Note here by Lemma  \ref{estable2}, we have at least one embedded component of an index at most $-4$ or at least two components of index $-2$.

Lastly, any stratum in Part $(2)$ has correct codimension in $\mA_u$, by the index computation for an embedded $(-2)$ sphere or by the existence of a section class in the correct codimension.

Indeed, if $J$ is in $\mA_u^2$,  either   there is exactly one exceptional curve $E$ breaks into $C+D$ where $C$ is the unique -2 curve and $D$ is another exceptional curve (and no embedded section classes of large codimension),  or there is an embedded curve in the class  $B+kF-\sum E_i$ of index -2.

\end{proof}

\begin{rmk}
\begin{itemize}
    \item  Lemma \ref{strsep}  states that $\mA^{high}_u$ is  a union of strata of Fr\'echet suborbifolds with codimension 4 or higher.
    \item  The changes in $\mA^{high}_u$ as $\omega $ varies are not well understood for general deformations as one does not have enough curves to perform inflation to reposition {\bf within} a given chamber. This is the main impediment to establishing the full stability conjecture in general case.
    \item  However,  $\mA^{high}_u$ is  well understood in the  $(\mu, 1, \frac12, \cdots \frac12)$ case since in this cases the exceptional curves cannot degenerate for homological reasons.
\end{itemize}

\end{rmk}

\subsection{Stability of $Symp(M,\w)$ and inflation}

 Firstly, note that one can always inflate along the curve in the class $F$, we have the following Lemma which allows us the find an inclusion between different $[\w]$.

\begin{lma}\label{incup}
For any stratum, including the open strata,   $\mA_{u, \mC}\subset \mA_{u', \mC}$, $u=(\mu,1,c), u'=(\mu+\eps,1,c)$,  and for all $\mu >1, \eps > 0$.

 \end{lma}
\begin{proof}
  By \cite{Zhang16} Theorem 1.6, we known that for each $J\in
\mA_{u, \mC}$, through each point of $M$ there is a stable $J$-holomorphic sphere representing the fiber class $F = [{\rm pt} \times S^2]$.

  Then we can inflate along the embedded curve $F$. Let us start with $u=(\mu, 1, c).$

 Inflating, we obtain a form in $t P.D[F] +(\mu, 1, c) $= $(\mu+t, 1, c)$, for all $t \in [0,\infty).$

\end{proof}

\begin{prp}\label{inflation}
In the following cases, for  $\mu>g$ the strata have inclusion relations:
\begin{enumerate}

\item    $\mA^{top}_{u}\supset \mA^{top}_{u'}$, where $u=(\mu,1, c_i), u'=(\mu+\eps,1,c_i)$,  and for all $\mu >max\{n, g\}, \eps > 0$.

 \item $  \mA^{2}_{u, C, D}   \supset   \mA^{2}_{u', C, D},$   $u=(\mu,1,c_i), u'=(\mu+\eps,1,c_i)$, for all $\mu >1, \eps > 0$.

\item For any  strata   $ \mA^{2}_{u, B+kF-\sum E_i} \subset    \mA^{2}_{u, B}$ where $B+kF-\sum E_i$ has index -2,    $ \mA^{2}_{u, B+kF-\sum E_i, }\supset \mA^{2}_{u', B+kF-\sum E_i}$  , we have $u=(\mu,1,c_i), u'=(\mu+\eps,1,c_i)$, for all $\mu >1, \eps > 0$.
 \end{enumerate}
 \end{prp}

\begin{table}[ht]
\begin{center}
\resizebox{\textwidth}{!}{\renewcommand{\arraystretch}{2}\begin{tabular}{||c c c c ||}
\hline\hline
Direction  &  Strata         &
Class to inflate & Proof   \\[0.5ex]
\hline\hline
 In the same chamber &                                                       $\mA_{u, \mC} $    & $B+xF-\sum E_i$,
$E_i$, $F-E_i$                                                 &  Lemma \ref{inflation}            \\[0.5ex]
\hline
Within a chamber &                                                       $\mA^2_{u, C, D} $    & $B+xF-\sum E_i$, {C, D},
$E_i$, $F-E_i$                                                 &  Lemma \ref{cod2inf}            \\[0.5ex]
\hline
     Within a chamber &     $\mA^{top}_{u} $  &
$B+xF-\sum E_i$, $E_i, F-E_i$                                            &  Lemma \ref{inflation}            \\[0.5ex]
\hline
Across to chambers with large $\mu$                                                       &  Any strata     &
$F$                                                      &  Lemma \ref{incup}            \\[0.5ex]
\hline
Across to chambers with small $\mu$                                                        &   $\mA_{u, \mC} $ and  $\mA^{top}_{u}  $    &
$B+xF-\sum E_i$                                                      &  Lemma \ref{inflation}            \\[0.5ex]
\hline\hline
\end{tabular} }
\caption{Inflation process}\label{inftable}
\end{center}
\end{table}

\begin{proof}

First, by Lemma \ref{strsep}, for any of the three types of strata considered in the current Lemma,  there is an embedded curve in class $A=B+xF-\sum E_i$, where $x<g$.

   Note that the process of inflating along this curve will increase the areas of the fiber class $F$ and the $E_j$, if $E_j \cdot A>0$. That means that while we land the desired chamber with a smaller section class $\mu$ we will not be positioned back on the horizontal line considered in the statements. To reposition within this smaller $\mu$ chamber we will use the exceptional classes curves.
   For the sets in points (1) and (3), since all $E_j$'s and $F-E_j$'s are embedded, we can perform simultaneous inflation along collections of such exceptional curves,  decreasing proportionally the area of the curves involved to infinitesimally reach the exterior wall simplices.  On the other hand, for the sets in point (2), we have to explain one additional inflation that sends towards the reduction simplicial wall ( Lemma \ref{cod2inf}.)  Because we have  are in a codimension 2 strata, there is at most one $E_j$ that has a type $\{C, D\}$ degeneration, then by Lemma \ref{cod2inf} we can still inflate along  the pair $C, D$ in such a manner that the homological effect on $
   u$ is the same as if the exceptional curves were embedded.

  More concretely, here is how  we inflate along a curve $B+xF -\sum E_i$.  Let's start with $u=(\mu, 1, c_i).$ By inflating the classes $F-E_i$ and $B+xF-\sum E_i$,  we obtain a form in
   $$t P.D[B+xF-\sum E_i] +(\mu, x, c_i) +t_i(1,0,c_i) = (\mu +tx+ \sum t_i, 1+t , t_i+c_i),$$
  which normalized to $$ \biggl(\frac{tx+\mu +\sum t_i}{1+t}, 1, \frac{c_i+t_i}{1+t}\biggr),$$
  $\forall t \in [0,\infty).$

   Note that we will choose $t_i/t$ proportional to $c_i$, that is, we always take $t_i=c_i t$.  Then in the resulting symplectic class, the area of the exceptional curves is always stable after normalization.
   
   When $B+xF-\sum E_i$ is a negative self intersection curve (note that for the current Lemma that can only happen for codimension 2, $g=1$, and $B+xF-\sum E_i$ has square $-1$), the computation is almost the same except the range of $t$,
   $$t P.D[B+xF-\sum E_i] +(\mu, x, c_i) +t_i(1,0,c_i) = (\mu +tx+ \sum t_i, 1+t , t_i+c_i),$$
  which normalized to $$ \biggl(\frac{tx+\mu +\sum t_i}{1+t}, 1, \frac{c_i+t_i}{1+t}\biggr),$$
  $\forall t \in [0, \w(B+xF-\sum E_i)).$
  
  This also grants that we move to any chamber with pseudo-holomorphic $B+xF-\sum E_i.$ 
   
   By inflating the classes $F-E_i$ and $B+xF-\sum E_i$,  we obtain a form in
   $$t P.D[B+xF-\sum E_i] +(\mu, x, c_i) +t_i(1,0,c_i) = (\mu +tx+ \sum t_i, 1+t , t_i+c_i),$$
  which normalized to $$ \biggl(\frac{tx+\mu +\sum t_i}{1+t}, 1, \frac{c_i+t_i}{1+t}\biggr),$$
  $\forall t \in [0,\infty).$

   Then we just need to make sure that as long as $t\to \infty$,  the resulting symplectic class covers  $\mu \ge g$ cases.  Note $\displaystyle{\lim_{t\to \infty }\frac{tx+\mu\sum t_i }{1+t} }=x\leq g.$   Hence we proved the statement of the Lemma.
 \end{proof}
As explained in the above lemma, repositioning on the lower strata inside a chamber is done by inflating along the embedded exceptional curves; the following deals with that in the case that $J$ admits $\{C, D\}$ type degeneration.

\begin{lma}\label{cod2inf}
  For any $J$ admitting a $\{C, D\}$ type  degeneration of an exceptional class $E$,and for a symplectic form $\omega$ compatible with $J$,  we can inflate  along the two irreducible components repeatedly in an alternate fashion so that we obtain a symplectic form $\w'$ tamed by $J$ and $[\w'] = [\w] +t P.D.(E), 0\le t < \w(E)-\w(C).$  In other word, 	there exists a symplectic form $\w_{\epsilon}$ such that the symplectic area of the exceptional classes $E$ is $\epsilon$ close to the area of $C$, for an arbitrary small $\epsilon$.

\end{lma}

\begin{proof}
Assume for a given $J$ compatible with $\w$, the exceptional class $E$ has a stable curve representative degenerating as in the homology type (2) in Lemma \ref{estable2};  i.e. the stable curve has two irreducible component classes $C$ and $D$. They each have an embedded representative and intersect each other transversely $C \cdot D=1$.

We will perform alternate inflation on the curves $D, C$; every step starts with a compatible form and ends with a tamed form, however, in between two steps we use cone Theorem of  \cite{LZ09} introduced in section \ref{tcinf}   to reposition on a compatible form in the same cohomology classes.

	Notice that each inflation step will affect all areas of  $ C, D, E$.

	\begin{itemize}
		\item Step 1: inflate along $D = (E-C)$, on an arbitrarily large  interval inside $0 \leq t <  \frac{\w(E)- \w(C)}{2} $ \footnote{ A convenient way of avoiding the $epsilon$ in each step of the J-tame inflation process is to use the formal inflation as defined in \cite{Zha17}. }. We obtain an   $\w_1$ (compatible with $J$ due to the above discussion) s.t. areas $  (  \w_1(E), \w_1(C) )$  is  arbitrarily close to  $( \frac{\w(E)+ \w(C)}{2}, \frac{\w(E)+ \w(C)}{2} ).$ 
		\item Step 1': inflate along $C$, on an arbitrarily large  interval inside $0 \leq t < \omega_1(C)= \frac{\w(E)- \w(C)}{2} $. As above, we pick  $\w^{'}_{1}$ compatible with $J$  s.t. areas $  (  \w^{'}_{1}(E), \w^{'}_{1}(C) )
		$  are  arbitrarily close to  $( \frac{\w(E)+ \w(C)}{2} =\w(C) +\frac12 (\w(E)-\w(C)), \w(C) ).$

	\item  Repeat the above two steps, one gets a symplectic form compatible with $J$  with area of the curves $E, C$ arbitrarily close to $( \frac{1}{4}( {\w(E)- \w(C)})+\w(C), \w(C) ).$

			\item   $\cdots$

			\item  Get a symplectic form $\omega^{'}_{N+1}$ compatible with $J$  with area of the curves $E, C$ arbitrarily close to $((\frac{1}{2})^N( {\w(E)- \w(C)})+\w(C), \w(C) ).$

	\item   $\cdots$
	\end{itemize}

This finishes the proof.
\end{proof}

\begin{rmk}
Let's use the following example to make the above process more explicit:  $E=E_1, C=E_2,$ and $D=E_1-E_2.$

	\begin{itemize}
		\item Step 1: inflate along $D = (E-C)$, one have $\w_1 \sim ( \mu, 1, (c_1+c_2)/2, (c_1+c_2)/2, \cdots ).$ 
		\item Step 1': inflate along $C$, one have $\w^{'}_{1}\sim( \mu, 1, (c_1+c_2)/2, c_2, \cdots ).$ 

	\item  Repeat the above two steps, one has $\w^{'}_{2}\sim( \mu, 1, c_1/4+3c_2/4, c_2, \cdots ).$ 

			\item   $\cdots$

			\item $\w^{'}_{K}\sim ( \mu, 1, c_1/(2^K)+c_2 (2^K-1)/(2^K),  c_2, \cdots ).$ 

	\item   $\cdots$
	\end{itemize}
\end{rmk}

Up to now, we have not used the condition $\mu >n$. However, in the event that there is some embedded section type embedded curve of index $-2$ the following Lemma guarantees optimal bounds for the corresponding negative inflation along this curve. Essentially this says that any nonempty  stratum of $\mA^{2}_{u, B}$ stays stable under inflation until reaching the optimal chamber.

\begin{lma} \label{type2}
Any curve in class $B+kF-\sum E_i$ with index $-2$ has a positive area when $\mu >n$.
\end{lma}

\begin{proof}
$B+kF-\sum E_i$ has codimension    $-4k +2g-2 +2l =2$  which means  index   $4k +2 -(2g +2l) =-2$ where $l$ is the number of $-E_i's$ in  $B+kF-\sum E_i$.     Since $l \ge 0, g \ge 1$ which means $4k \ge -2$,     and now we have $k\ge 0$ because $k$ is an integer.     The curve class $B+kF-\sum E_i$ has symplectic area  $\mu+k-\sum c_i>0$  if $\mu>n.$
\end{proof}

\begin{cor}\label{strconst}
For any  $ u $ with $\mu >max\{g, n\}$,

\begin{enumerate}
       \item  $\mA_{u}^{top}$ is constant for any such  $u$ within and outside any chambers.  \item  $\mA_{u}^2 $ is constant for any such $u$ within and outside any chambers.  \item $\mA_{u}^{high}\subset \mA_{u'}^{high}$, $u=(\mu,1,c), u'=(\mu+\eps,1,c)$,  and for all $\mu >1, \eps > 0$

\end{enumerate}

\end{cor}

\begin{proof}

This is a straightforward combination of Lemma \ref{inflation}, Lemma \ref{cod2inf} and Lemma \ref{type2}.  Note that the condition on $\mu$ in the current statement comes from $\mu>g$  of Lemma \ref{inflation} and $\mu >n$ in Lemma \ref{cod2inf}.

\end{proof}

\begin{prp} [=Theorem \ref{stab01intro} ]\label{stab01}
Suppose $\mu, \mu'>max\{g, n\}$ for $u= (\mu, 1, c_1,\cdots, c_n)$, $u'= (\mu', 1, c'_1,\cdots, c'_n)$,  then the groups  $\pi_0$ and $\pi_1$ of $G^g_{u,n}   $ and $   G^g_{u',n}  $ are  the same.

\end{prp}

\begin{proof}
First let us assume that  $u$ and $u'$ are on a horizontal line $u=(\mu,1,c), u'=(\mu+\eps,1,c)$,  and for all $\mu >1, \eps > 0$
Proposition \ref{inflation} allows us to write the right-hand side inclusions in the following diagrams
 
\begin{equation}\label{hcomm}
\begin{array}{ccccccc}
{(a)} & & G^g_{u,n} &\to & \Diff_0(M_g \# n \overline{ \CC P^2}) & \to & \mA_{u}\\
& & \downarrow & & \;\downarrow = & & \downarrow\\
& & G^g_{u',n} &\to & \Diff_0(M_g \# n\overline{ \CC P^2}) & \to & \mA_{u'},\\
& & & & & & \\
{ (b)} & & & G^g_{u,n}  \rightarrow G^g_{u',n} & & \\
& & &  \searrow \downarrow & & \\
& & &  G_{u''}& &
\end{array}
\end{equation}
 
 Along with Corollary \ref{strconst} this is sufficient to get the results along horizontal lines. Incidentally, these diagrams also allow us to state that the homotopy colimit exists on horizontal lines as $\epsilon  \rightarrow  \infty$.
 
 The same corollary implies that the open strata and the codimension $2 $ strata are constant {\it within the chambers}, thus   the result can be established in full generality for all $u$ within the conditions of the theorem.

\end{proof}

Now let's recall the following result from \cite[Proposition 5.1]{McD08}:

 \begin{thm}\label{csympk}
  Let $(M,\w)$ be a  symplectic manifold  and denote by $M_n$ its $n$-fold blow  up.  Then:\smallskip
  
 {\rm  (i)}  There is a homomorphism $f_*^n: 
  \bigl(\ZZ \oplus \pi_2M\bigr)^n \to \pi_2\bigl(B\Diff(M_n)\bigr)$ whose kernel is 
contained in the torsion subgroup of $\bigl(\pi_2M\bigr)^n$.\smallskip

{\rm (ii)}    there is $\eps_0>0$ such that
the elements in $f_*^n\bigl((\pi_2M)^n\bigr)$ can all be realised in
$BSymp(M_n, \om_\eps)$ whenever the blow up parameter $\eps = (\eps_1,\dots,\eps_n)$ satisfies  $\eps_i\le \eps_0$ for all $i$.
\end{thm}

Applying Theorem \ref{csympk}  to our case where $M_g \# n\overline{ \CC P^2}$  is the blow up ruled surface $M_g$:

The torsion free part of $\pi_2(M_g \# n\overline{ \CC P^2})$ is  $\ZZ$, and hence by (i), $f_*^n\bigl((\pi_2(M_g \# n\overline{ \CC P^2})^n\bigr)$ has torsion-free part at least $\ZZ^n$. By (ii) those can all be realized in $\pi_1G^g_{u,n}$ whenever the vector $  u  $  comes from small enough blow up sizes $c_i$.

\subsection{Stability of equal size 1/2}\label{1/2}

We can prove the stronger version of the stability result for equal size 1/2.  Here we first describe the chamber structure and then provide a proof.

Notice that the space of such forms in class $(\mu, 1, \frac12,\cdots, \frac12), \mu>g$ is a line.  The positive codimension strata are  given exclusively by embedded curves in classes of they type $B-kF$  or $B-kF-\sum E_i.$  This is because since each $E_i$ and $F-E_j$ has area $\frac12,$ there is no degeneration of exceptional curves.  This means that the entries  $\mA^2_{u, C, D}$ or $\mA_u^{high}$ in the first row of table \ref{partable} are actually empty.  Hence on the line  $(\mu, 1, \frac12,\cdots, \frac12)$,  the chambers are guided by those values of $\mu$ that are either integer points or half integer points where curves $B-kF$  or $B-kF-\sum E_i$ changes.

Hence the precise statement is the following:
\begin{thm}[=Theorem \ref{stabintro}]

The homotopy type of $G^g_{\mu,n}$ is constant for $\frac{k}{2} < \mu \le \frac{k+1}{2},$ for
any integer $k \ge max\{2g, n\}$. Moreover as $\mu$ passes the half integer $\frac{k+1}{2}$, all the groups
$\pi_i, i = 0,\cdots,2k + 2g - 1$ do not change.

\end{thm}

\begin{proof}

The proof is basically the same as that in Proposition \ref{inflation}.   The only thing added here is the inflation along  the embedded curves in classes $B-kF-\sum E_i.$ This provide inclusions of  the higher codimensional strata    $\mA_{u', B-kF-\sum E_i} \subset \mA_{u, B-kF-\sum E_i} $ for $u=(\mu,1,c), u'=(\mu+\eps,1,c)$. Along with Corollary \ref{strconst} which provides the opposite inclusions, this provides all conditions to implement   Proposition \ref{stab01} yielding the result.
\end{proof}

\begin{cor}
 There are new elements in $\pi_1(G^g_{u,n})$ for a symplectic form without any Hamiltonian circle actions.
\end{cor}

\begin{proof}
Notice that by Theorem \ref{csympk}, on this line of equal size $1/2$, the rank of $\pi_1(Symp(M,\w))$ is positive, indeed at least $k$ for k-fold blowups.   On the other hand, by \cite[Corollary 3.6]{HK15}, if the number of blowup is large ($n \ge \mu/2$), there are points on this line that do admit Hamiltonian circle actions.
\end{proof}

\section{Singular foliations and topological colimit for equal size blowups}\label{s:out}

As mentioned in the proof of the  stability Theorem \ref{stab01} the homotopy colimit 
exists along each horizontal line fixing the blowup size.

 In what follows we are going to investigate that homotopy colimit in the special case when the horizontal line is along equal size blow-ups of increasing base area.

\subsection{The equal size $\frac12$ blowup}\label{ss:bi}

In this special case, we are going to use the relationship between the space of singular foliations and the space of almost complex structures to establish a smooth diffeomorphism model for this homotopy colimit denoted in this case by $G^g_{\infty,n}$. We will show that this smooth diffeomorphism model is disconnected and hence conclude that  $G^g_{\infty,n}$   is disconnected.

 Proposition \ref{incup}  and the  homotopy commutative diagrams \eqref{hcomm} shows that the homotopy colimit exists. Additionally, from  Theorem \ref{stabintro},
we have that as $\mu\to \infty$,  $\pi_iG^g_{\mu,n}=   \pi_iG^g_{\infty,n}$ for $i \leq 2k + 2g - 1.$    We also introduce the space $\mA_{\infty}$ of almost complex structures, as a colimit of $\mA_{u}$'s as $\mu \to \infty,$  namely, $\mA_{\infty} := \underset{\mu>max\{g,n\}}{\cup}{\mA_u}.$

In what follows, we will introduce a space of foliations that can be used to correctly identify a smooth model for our colimit. 
We are going to use the following singular foliation in the smooth (topological) category:

\begin{dfn}\label{singfol}
A {\bf n-singular foliation} by $S^2$ of $M_g \# \overline{ \CC P^2}$ is defined as a foliation with smooth embedded spherical leaves in the $F=[pt \times S^2]$ class
and n nodal leaves with two embedded spherical components, each in the class
$E_i, 1\le i \le n$ and $F-E_i$ where $E_i$'s are the exceptional classes. Also, we require that the complement of the singular leaf is a smooth foliation over an n-punctured genus g surface $Y$.   We denote ${  Fol}$  by the space of such n-singular foliations.
\end{dfn}

As argued in Section \ref{1/2}, any exceptional curve has the minimal area in the equal size $\frac12$ blowup case and they can never degenerate.   For each  $J$ taming a symplectic form in a class $u=(\mu, 1, \frac 12, \cdots \frac 12)$ there exists a singular foliation as in Definition \ref{singfol} where the leaves are actually $J-$ holomorphic.

  Let $\Ff_{std}$ be the standard blow-up foliation by $J_{std}$-spherical leaves.
  Note that if we blow down the complex structure, we obtain the split complex structure on $M_g  $ and the induced foliation is the split foliation by the spheres.

Following verbatim the argument in \cite{BL1}, we have the following Lemma on the space of foliations and transitive action, when there is only finitely many nodal fibers.

\begin{lma}
 Let $\Fol_0$ be the connected component of $\Fol$ that contains $\mF_{std}$. Then the colimit $\mA_\infty $ is weakly homotopic to $\Fol_0$.
\end{lma}

\begin{proof}

Observe that there is a map $\mA_\infty \to \Fol_0$ given by taking
$J$ to the singular foliation of $M_g \# n \overline{ \CC P^2}$ by $J$-spheres in class $F$ or $F-E$.  Standard arguments
in \cite{MS17} Ch 2.5 show that this map is a
fibration with contractible fibers.   Hence it is a homotopy equivalence.\\
\end{proof}

\begin{lma}\label{tranfol}
There is a transitive action of $\Diff_0(M_g \# n \overline{ \CC P^2})$ on $\Fol_0$. 
\end{lma}

\begin{proof}

Since $S^2\setminus pt$ is compact and
 simply connected, each generic leaf of this foliation has trivial holonomy and
hence has a neighborhood that is diffeomorphic to the product $D^2\times
S^2$ is equipped with the trivial foliation with leaves $pt\times S^2$.

 Since our foliation has smoothly embedded leaves and only one nodal leaf, we can find a 2-form transverse to each leaf. Moreover, the Poincar\'e dual of such 2-form is a smooth section, not passing through the singular point $p$.

 Now let's take an arbitrary singular foliation $\mF^{'} \in \Fol_0$ and  denote the smooth section by $\Sigma^{'}.$  We'll prove that $\Diff_0(M_g \# n \overline{ \CC P^2})$ takes this foliation $(\mF^{'}, \Sigma^{'})$ to $\mF_{std}, \Sigma_{std}$ where $\Sigma_{std}$ is the smooth section (which is indeed $J_{std}$-holomorphic).

 Since $\mF'$ and $\mF_{std}$ are in the same path connected component, there is a $\phi \in  \Diff_0(M_g \# n \overline{ \CC P^2})$ sending  $\Sigma^{'}$ to $ \Sigma_{std},$ such that the singular leaf of $\mF^{'}$ goes to the singular leaf of $\mF_{std}$ while the two singular points are identified.  Now let's fix a finite covering $\{D_i, 1\leq i \leq n\}$ of $\Sigma^{'},$ such that the local foliations over $D_i$'s cover the manifold $\Sigma_g\times S^2 \# \overline{ \CC P^2}.$

 Then we use partition of unity for the covering  $\{D_i, 1\leq i \leq n\}$ of $\Sigma^{'},$ and for each local foliation, we apply a $\phi_i$ such that the foliation  $\mF^{'}$ under $\phi\circ \phi_1 \circ \cdots \circ \phi_n$ agrees with the foliation $\mF_{std}.$

 Now we have the transitive action of $\Diff_0(M_g \# n \overline{ \CC P^2})$ on $\Fol_0$.  Notice that this action of $\Diff_0(M_g \# n \overline{ \CC P^2})$ does not necessarily preserve the regular leaves. However, it must stay invariant on each of the singular leaves and preserve the singular points and the intersections of the singular fibers with the base since any diffeomorphism isotopic to the identity acts trivially on homology.
\end{proof}

  Hence  there is an orbit-stabilizer fibration  from the transitive action, where the isotropy group is described in the following definition:

\begin{dfn}\label{fibergp}
$\Dd^g_n$ consists of all elements in the identity component of the diffeomorphisms group which fit into the commutative diagram $$
\begin{array}{ccc} M_g \# n \overline{ \CC P^3}& \stackrel \phi{\to} &M_g\# n \overline{ \CC P^2}\\
\downarrow & &\downarrow\\
  M_g, \{p_1,\cdots, p_n\}, F_p& \stackrel {\phi'} {\to} & M_g, \{p_1,\cdots, p_n\}, F_p\\
\downarrow & &\downarrow\\
  (\Sigma_g, x_1, \cdots, x_n) & \stackrel {\phi''}{\to}  & (\Sigma_g,x_1, \cdots, x_n).
\end{array}
$$

\end{dfn}

Here $p_i$ is the intersection point $E_i\cap (F-E_i)$ of the singular fiber. And the first level of the downward arrow means that we contract the $E_i$ component. We abuse notation here to still denote $p_i$ for the point in $M_g$ after contracting the curve $E_i$.

On the second level, $\phi'$ is a diffeomorphism of $M_g$ keeping the points $p_i$ fixed and fixing the fiber $F_p$ passing through $p_i$ fixed as a set, and preserves other leaves in the standard foliation.

The base $\Sigma_g$ is the holomorphic curve $B_{std}$ w.r.t the standard complex structure, and the map downward is obtained by firstly blowing down the exceptional sphere and then projecting down to the base curve.

From the above definition, it is clear that the group element of $
\Dd^g_n \subset \Diff_0(M_g\# \overline{ \CC P^2})$ preserves the leaves setwise in $Fol_0$ and hence we have the orbit-stabilizer associated to Lemma \ref{tranfol} being

\begin{equation}\label{folpre}
 \Dd^g_n \to
\Diff_0(M_g \# n\overline{ \CC P^2}) \to \Fol_0.
\end{equation}

\begin{prp}\label{tlimit}
  Take $M_g \# n \overline{ \CC P^2}$ with a form in the class $(\mu,1, \frac12,\cdots,\frac12)$. Then let $\mu$ go to $\infty$.
\begin{enumerate}
\item $\Dd^g_n$ is weakly homotopic to  $G^g_{\infty,n}$.

\item The group $\Dd^g_n$ is disconnected when $g\ge 2$.

\item When $\mu\to \infty$, s.t. for $i=0,1$, $\pi_i(G^g_{u,n})=   \pi_i(G^g _{\infty,n})$ for $i \leq 2k+2g-1,$ and hence the groups $G^g_{u,n}$ are disconnected for $g \geq 2$.
\end{enumerate}

\end{prp}

\begin{proof}
For statement (1), note the equation \eqref{folpre} fits into the
commutative diagram:
$$
\begin{array}{ccc}
  \Diff_0(M_g \#n \overline{ \CC P^2})  & \to  & \mA_\infty\\
\downarrow & & \downarrow\\
\Diff_0(M_g \# n \overline{ \CC P^2}) & \to & \Fol_0,
\end{array}
$$
where the upper map is given as before by the action $\phi\mapsto \phi_*(J_{{\rm std}})$.  Hence there is an induced homotopy equivalence from the homotopy fiber  $G^g_{\infty,n}$  of the top row to the fiber
$\Dd^g_n$ of the second.

To prove statement (2), first note that we have the following  fibration
  \[  \Diff(\Sigma_g, x_1,\cdots,x_n) \longrightarrow \Diff(\Sigma_g) \longrightarrow Conf(\Sigma_g,n),\]

  Where  $Conf(\Sigma_g,n)$ is the configuration of ordered n points on $\Sigma_g$.

 Taking the right portion of the long exact sequence, we have:

  \[ 1 \longrightarrow \pi_1(Conf(\Sigma_g,n)) \longrightarrow \pi_0[\Diff(\Sigma_g, x_1,\cdots,x_n)]  \longrightarrow \pi_0[\Diff( \Sigma_g)]\longrightarrow 1.\]

  Then restricting to $ \Diff_0( \Sigma_g) $, we obtain an element in the identity component of $\pi_0(\Diff(\Sigma_g))$ but not in the identity component of $\pi_0(\Diff(\Sigma_g, x_1,\cdots,x_n))$, where $x_1,\cdots,x_n$ are the points we will blow-up.

It can be constructed explicitly in the following way: choose a path $\alpha(t)\subset  \Diff(\Sigma_g), t \in [0, 2\pi]$, pushing $x_1,\cdots,x_n \in \Sigma_g$ along   homologically non-trivial loop in $Conf(\Sigma_g,n).$
 Now $\alpha(0)=id$ and $\alpha(2\pi) \in \Diff(\Sigma_g, x_1,\cdots,x_n) \cap \Diff_0( \Sigma_g) $  and note that $\alpha(2\pi) $ is  the desired element.

 Next, we lift the path $\alpha(t)$  into dimension 4. To do that, first fix $M_g$, $\Sigma_g$ and choose $J_{split}$. There is a natural family $\alpha(t)\times id \subset \Diff_0(M_g),$ which act on the leaves  in the trivial manner.  For each of $t$, we have a product complex structure on $M_g$ by pulling back  $J_{split}$ by $\alpha(t)\times id$. Recall that $p_i$'s are the preimage of $x_i$'s under the projection $M_g\to \Sigma_g$. We are going to obtain a family of complex structures by blowing up at the points $(\alpha(t) \times id )|_{p_i} \in M_g,  1\le i \le n.$   This gives us a loop of complex structures $J_t$ on $M_g \# n \overline{ \CC P^2}$  where $J_0=J_{{\rm std}}$. Note that by \cite{Zhang16},  each $J_t$   gives rise to a singular foliation $\mF_t$, as in Definition \ref{singfol}. Geometrically, $\mF_t$ is a loop in $Fol_0$ starting with the standard singular foliation $\mF_{std},$ 
  pushing $x_1,\cdots,x_n \in \Sigma_g$ along   homologically non-trivial loop in $Conf(\Sigma_g,n)$ for each time $t\in[0,2\pi]$.

By the transitivity Lemma \ref{tranfol}, we can use a path $\phi_t$ in $\Diff_0(M_g \# n\overline{ \CC P^2})$ to push $\mF_0,$ so that $\phi_t\circ \mF_0= \mF_t.$    Note that $\phi_t$ in $\Diff_0(M_g \# n \overline{ \CC P^2})$ pushes the standard foliation along this loop.

 Now we focus on the diffeomorphism $\phi_{2\pi}$.  First note that $\phi_{2\pi}$ preserves the singular foliation $\mF_{std},$ since the foliation $\mF_{2\pi} =\mF_0=\mF_{std}.$  Hence  $\phi_{2\pi} \in \mD^g_n$.   Also, the above paragraph gives an explicit isotopy of $\phi_{2\pi}$ to the identity map in  $\Diff_0(M_g \# n\overline{ \CC P^2})$, through the path $\phi_t.$

 Finally,we show that $\phi_{2\pi}$ is not isotopic to id in $\mD^g_n$. Suppose there is an isotopy to id, then by path lifting of the fibration \ref{folpre},  we would have a leaf-preserving element in $\Diff_0(M_g \# n\overline{ \CC P^2})$, so that it is isotopic to identity through a path in $\mD^g_n$. Furthermore, this path pushes the given foliation along the lifting of the loop $\mF_t, t\in[0,2\pi]$.  Now apply the diagram of definition \ref{fibergp}. We would have an isotopy that would in turn give an isotopy of $(\Sigma_g, x_1,\ldots,x_n)$, connecting the time $2\pi$ diffeomorphism to identity.  This is a contradiction against the Birman exact sequence. Hence statement (2) holds.

Statement (3) follows from the stability Theorem \ref{stab01}.
\end{proof}

\begin{rmk}
  When $g=0$, one can blow up $S^2\times S^2$ at $n$ points with equal sizes. It is shown in \cite{LLW15} that when $n\le 3$, $ G^0_{u, n}$ is connected for all $u$. When $n>3$, $ G^0_{u, n}$  (   blowup equal and 1/2 of the size of the fiber ) is a braid group of $n$ strands on spheres (cf. \cite{LLW3}).  This follows the same pattern as  $\Diff(S^2,n)$, which is the diffeomorphism group of $S^2$ fixing $n$ points. Their examples are constructive elements in $\pi_0 G^0_{u,n} $ produced using ball swapping techniques. As pointed out in Example 2.3 of \cite{LWnote}, there is a way to construct ball swappings of a ball along a non-trivial loop in $\Sigma_g$.  It is an interesting question to explore whether the construction here is indeed a ball swapping map. An initial question in this direction is whether ( using either construction) one can see if  $\pi_0 \mD^g_n$ is a braid group of n strands on $\Sigma_g$. \\
\end{rmk}

\nocite{}
\printbibliography

\end{document}